\documentclass{amsart}
\usepackage{amssymb}

\newtheorem{theorem}{Theorem}[section]
\newtheorem{lemma}[theorem]{Lemma}
\newtheorem{result}[theorem]{Result}
\newtheorem{corollary}[theorem]{Corollary}
\theoremstyle{definition}
\newtheorem{definition}[theorem]{Definition}

\theoremstyle{remark}
\newtheorem{remark}[theorem]{Remark}

\numberwithin{equation}{section}

\begin{document}

\title{Certain Transformations for Hypergeometric series in $p$-adic setting}


%    author one information
\author{Rupam Barman}
\address{Department of Mathematics, Indian Institute of Technology, Hauz Khas, New Delhi-110016, INDIA}
\curraddr{}
\email{rupam@maths.iitd.ac.in}
\thanks{}

%    author two information
\author{Neelam Saikia}
\address{Department of Mathematics, Indian Institute of Technology, Hauz Khas, New Delhi-110016, INDIA}
\curraddr{}
\email{nlmsaikia1@gmail.com}
\thanks{}

%    \subjclass is required.
\date{10th March, 2014}

%    Abstract is required.
\begin{abstract} In \cite{mccarthy2}, McCarthy defined a function $_{n}G_{n}[\cdots]$
using the Teichm\"{u}ller character of finite fields and quotients of the $p$-adic gamma function.
This function extends hypergeometric functions over finite fields to the $p$-adic setting.
In this paper, we give certain transformation formulas for the function $_{n}G_{n}[\cdots]$
which are not implied from the analogous hypergeometric functions over finite fields.
\end{abstract}
\maketitle
\section{Introduction and statement of results}
In \cite{greene}, Greene introduced the notion of hypergeometric functions over finite fields or
\emph{Gaussian hypergeometric series}. He established these functions as analogues of classical hypergeometric
functions. Many interesting relations between special values of Gaussian hypergeometric series and the number of
points on certain varieties over finite fields have been obtained. By definition, results involving hypergeometric functions over
finite fields are often restricted to primes in certain congruence classes. For example, the expressions for the trace
of Frobenius map on certain families of elliptic curves given in \cite{BK1, BK2, Fuselier, lennon, lennon2} are restricted to
such congruence classes.
In \cite{mccarthy2}, McCarthy defined a function
$_{n}G_{n}[\cdots]$ which can best be described as an analogue of hypergeometric series in the $p$-adic setting.
He showed how results involving Gaussian hypergeometric series can be extended to a wider class of primes using the function
$_{n}G_{n}[\cdots]$.
\par
Let $p$ be an odd prime, and let $\mathbb{F}_q$ denote the finite field with $q$ elements, where $q=p^r, r\geq 1$.
Let $\phi$ be the quadratic character on $\mathbb{F}_q^{\times}$ extended to all of $\mathbb{F}_q$ by setting $\phi(0):=0$.
Let $\mathbb{Z}_p$ denote the ring of $p$-adic integers.
Let $\Gamma_p(.)$ denote the Morita's $p$-adic gamma function, and let $\omega$ denote the
Teichm\"{u}ller character of $\mathbb{F}_q$. We denote by $\overline{\omega}$ the inverse of $\omega$.
For $x \in \mathbb{Q}$ we let $\lfloor x\rfloor$ denote the greatest integer less than
or equal to $x$ and $\langle x\rangle$ denote the fractional part of $x$, i.e., $x-\lfloor x\rfloor$.
Also, we denote by $\mathbb{Z}^{+}$ and $\mathbb{Z}_{\geq 0}$
the set of positive integers and non negative integers, respectively. The definition of the function $_{n}G_{n}[\cdots]$ is
as follows.
\begin{definition}\cite[Definition 5.1]{mccarthy2} \label{defin1}
Let $q=p^r$, for $p$ an odd prime and $r \in \mathbb{Z}^+$, and let $t \in \mathbb{F}_q$.
For $n \in \mathbb{Z}^+$ and $1\leq i\leq n$, let $a_i$, $b_i$ $\in \mathbb{Q}\cap \mathbb{Z}_p$.
Then the function $_{n}G_{n}[\cdots]$ is defined by
\begin{align}
&_nG_n\left[\begin{array}{cccc}
             a_1, & a_2, & \ldots, & a_n \\
             b_1, & b_2, & \ldots, & b_n
           \end{array}|t
 \right]_q:=\frac{-1}{q-1}\sum_{j=0}^{q-2}(-1)^{jn}~~\overline{\omega}^j(t)\notag\\
&\times \prod_{i=1}^n\prod_{k=0}^{r-1}(-p)^{-\lfloor \langle a_ip^k \rangle-\frac{jp^k}{q-1} \rfloor -\lfloor\langle -b_ip^k \rangle +\frac{jp^k}{q-1}\rfloor}
 \frac{\Gamma_p(\langle (a_i-\frac{j}{q-1})p^k\rangle)}{\Gamma_p(\langle a_ip^k \rangle)}
 \frac{\Gamma_p(\langle (-b_i+\frac{j}{q-1})p^k \rangle)}{\Gamma_p(\langle -b_ip^k \rangle)}.\notag
\end{align}
\end{definition}
The aim of this paper is to explore possible transformation formulas for the function $_{n}G_{n}[\cdots]$.
In \cite{mccarthy2}, McCarthy showed that transformations for hypergeometric functions over finite fields can be
re-written in terms of  $_{n}G_{n}[\cdots]$. However, such transformations will hold for
all $p$ where the original characters existed over $\mathbb{F}_p$, and hence restricted to primes in certain
congruence classes. In the same paper, McCarthy posed an interesting question
about finding transformations for $_{n}G_{n}[\cdots]$ which exist for all but finitely many $p$. In \cite{BS1}, the authors
find the following two transformations for the function $_{n}G_{n}[\cdots]$ which exist for all prime $p > 3$.
\begin{result}\cite[Corollary 1.5]{BS1}\label{cor1}
 Let $q=p^r$, $p>3$ be a prime. Let $a, b \in \mathbb{F}_q^{\times}$ and $-\dfrac{27b^2}{4a^3}\neq 1$. Then
 \begin{align}
&{_2}G_2\left[ \begin{array}{cc}
              \frac{1}{4}, & \frac{3}{4} \\
              \frac{1}{3}, & \frac{2}{3}
            \end{array}|-\dfrac{27b^2}{4a^3}
 \right]_q\notag\\
 &=\left\{
                                  \begin{array}{ll}
                                    \phi(b(k^3+ak+b))\cdot {_2}G_2\left[ \begin{array}{cc}
                                                         \frac{1}{2}, & \frac{1}{2} \\
                                                         \frac{1}{3}, & \frac{2}{3}
                                                       \end{array}|-\dfrac{k^3+ak+b}{4k^3}\right]_q \hbox{if~ $a=-3k^2$;}\\
                                    \phi(-b(3h^2+a))\cdot {_2}G_2\left[ \begin{array}{cc}
              \frac{1}{2}, & \frac{1}{2} \\
              \frac{1}{4}, & \frac{3}{4}
            \end{array}|\dfrac{4(3h^2+a)}{9h^2}
 \right]_q  \hbox{if ~$h^3+ah+b=0$.}
                                  \end{array}
                                \right.\notag
\end{align}
\end{result}
Apart from the transformations which can be implied from the hypergeometric
functions over finite fields, the above two transformations are the only transformations for the function $_{n}G_{n}[\cdots]$ in full
generality to date. In this paper, we prove two more such transformations which are given below.
\begin{theorem}\label{MT1}
 Let $q=p^r$, $p>3$ be a prime. Let $m=-27d(d^3+8)$, $n=27(d^6-20d^3-8)$ $\in \mathbb{F}_q^{\times}$ be such that $d^3\neq 1$, and
 $-\dfrac{27n^2}{4m^3}\neq 1$. Then
\begin{align}
&q\phi(-3d)\cdot {_2}G_2\left[ \begin{array}{cc}
 \frac{1}{2}, & \frac{1}{2} \\
 \frac{1}{6}, & \frac{5}{6}
 \end{array}|\dfrac{1}{d^3}\right]_q\notag\\
&=\alpha-q+\phi(-3(8+92d^3+35d^6))+q\phi(n)\cdot{_2}G_2\left[ \begin{array}{cc}
              \frac{1}{4}, & \frac{3}{4} \\
              \frac{1}{3}, & \frac{2}{3}
            \end{array}|-\dfrac{27n^2}{4m^3}
 \right]_q,\notag
 \end{align}
 where $\alpha=\left\{
           \begin{array}{ll}
             5-6\phi(-3), & \hbox{if~ $q\equiv 1\pmod{3}$;} \\
             1, & \hbox{if~ $q\not\equiv 1\pmod{3}$.}
           \end{array}
         \right.$
 \end{theorem}
Combining Result \ref{cor1} and Theorem \ref{MT1}, we have another four such transformations for the function
$_{n}G_{n}[\cdots]$ which are listed below.
\begin{corollary}\label{cor2}
 Let $q=p^r$, $p>3$ be a prime.
 Let $\alpha$ be defined as in Theorem \ref{MT1}, and $m=-27d(d^3+8)$, $n=27(d^6-20d^3-8)\in \mathbb{F}_q^{\times}$ be such that $d^3\neq 1$ and
 $-\dfrac{27n^2}{4m^3}\neq 1$.
 \begin{enumerate}
  \item
If $3k^2+m=0$, then
\begin{align}
&q\phi(-3d)\cdot {_2}G_2\left[ \begin{array}{cc}
\frac{1}{2}, & \frac{1}{2} \\
\frac{1}{6}, & \frac{5}{6}
\end{array}|\dfrac{1}{d^3}\right]_q\notag\\
&=\alpha-q+\phi(-3(8+92d^3+35d^6))+q\phi(k^3+mk+n)\notag\\
&~\times{_2}G_2\left[ \begin{array}{cc}
              \frac{1}{2}, & \frac{1}{2} \\
              \frac{1}{3}, & \frac{2}{3}
            \end{array}|-\dfrac{k^3+mk+n}{4k^3}
 \right]_q.\notag
\end{align}
\item If $h^3+mh+n=0$, then
\begin{align}
&q\phi(-3d)\cdot {_2}G_2\left[ \begin{array}{cc}
\frac{1}{2}, & \frac{1}{2} \\
\frac{1}{6}, & \frac{5}{6}
\end{array}|\dfrac{1}{d^3}\right]_q\notag\\
&=\alpha-q+\phi(-3(8+92d^3+35d^6))+q\phi(-3h^2-m)\notag\\
&~\times{_2}G_2\left[ \begin{array}{cc}
              \frac{1}{2}, & \frac{1}{2} \\
              \frac{1}{4}, & \frac{3}{4}
            \end{array}|\dfrac{4(3h^2+m)}{9h^2}
 \right]_q.\notag
\end{align}
\end{enumerate}
\end{corollary}
For an elliptic curve $E$ defined over $\mathbb{F}_q$,
the trace of Frobenius of $E$ is defined as $a_q(E):=q+1-\#E(\mathbb{F}_q)$, where $\#E(\mathbb{F}_q)$ denotes the number of $\mathbb{F}_q$-
points on $E$ including the point at infinity. Also, $j(E)$ denotes the $j$-invariant of $E$.
We now state a result of McCarthy which will be used to prove our main results.
\begin{theorem}\cite[Theorem 1.2]{mccarthy2}\label{mc}
Let $p>3$ be a prime. Consider an elliptic curve $E_s/\mathbb{F}_p$ of the form $E_s: y^2=x^3+ax+b$ with $j(E_s)\neq 0, 1728$. Then
\begin{align}
 a_p(E_s)=\phi(b)\cdot p\cdot {_2}G_2\left[ \begin{array}{cc}
              \frac{1}{4}, & \frac{3}{4} \\
              \frac{1}{3}, & \frac{2}{3}
            \end{array}|-\frac{27b^2}{4a^3}
 \right]_p.
\end{align}
\end{theorem}
\begin{remark} McCarthy proved Theorem \ref{mc} over $\mathbb{F}_p$
and remarked that the result could be generalized for $\mathbb{F}_q$.
We have verified that Theorem \ref{mc} is also true for $\mathbb{F}_q$.
We will apply Theorem \ref{mc} for $\mathbb{F}_q$ to prove our results.
\end{remark}

\section{Preliminaries}
Let $\widehat{\mathbb{F}_q^\times}$ denote the set of all multiplicative characters $\chi$ on $\mathbb{F}_q^{\times}$.
It is known that $\widehat{\mathbb{F}_q^\times}$ is a cyclic group of order $q-1$
under the multiplication of characters: $(\chi\psi)(x)=\chi(x)\psi(x)$, $x\in \mathbb{F}_q^{\times}$.
The domain of each
$\chi \in \mathbb{F}_q^{\times}$ is extended to $\mathbb{F}_q$ by setting $\chi(0):=0$ including the trivial character $\varepsilon$.
We now state the \emph{orthogonality relations} for multiplicative characters in the following lemma.
\begin{lemma}\emph{(\cite[Chapter 8]{ireland}).}\label{lemma2} We have
\begin{enumerate}
\item $\displaystyle\sum_{x\in\mathbb{F}_q}\chi(x)=\left\{
                                  \begin{array}{ll}
                                    q-1 & \hbox{if~ $\chi=\varepsilon$;} \\
                                    0 & \hbox{if ~~$\chi\neq\varepsilon$.}
                                  \end{array}
                                \right.$
\item $\displaystyle\sum_{\chi\in \widehat{\mathbb{F}_q^\times}}\chi(x)~~=\left\{
                            \begin{array}{ll}
                              q-1 & \hbox{if~~ $x=1$;} \\
                              0 & \hbox{if ~~$x\neq1$.}
                            \end{array}
                          \right.$
\end{enumerate}
\end{lemma}
\par Let $\mathbb{Z}_p$ and $\mathbb{Q}_p$ denote the ring of $p$-adic integers and the field of $p$-adic numbers, respectively.
Let $\overline{\mathbb{Q}_p}$ be the algebraic closure of $\mathbb{Q}_p$ and $\mathbb{C}_p$ the completion of $\overline{\mathbb{Q}_p}$.
Let $\mathbb{Z}_q$ be the ring of integers in the unique unramified extension of $\mathbb{Q}_p$ with residue field $\mathbb{F}_q$.
We know that $\chi\in \widehat{\mathbb{F}_q^{\times}}$ takes values in $\mu_{q-1}$, where $\mu_{q-1}$ is the group of
$(q-1)$-th root of unity in $\mathbb{C}^{\times}$. Since $\mathbb{Z}_q^{\times}$ contains all $(q-1)$-th root of unity,
we can consider multiplicative characters on $\mathbb{F}_q^\times$
to be maps $\chi: \mathbb{F}_q^{\times} \rightarrow \mathbb{Z}_q^{\times}$.
\par We now introduce some properties of Gauss sums. For further details, see \cite{evans}. Let $\zeta_p$ be a fixed primitive $p$-th root of unity
in $\overline{\mathbb{Q}_p}$. The trace map $\text{tr}: \mathbb{F}_q \rightarrow \mathbb{F}_p$ is given by
\begin{align}
\text{tr}(\alpha)=\alpha + \alpha^p + \alpha^{p^2}+ \cdots + \alpha^{p^{r-1}}.\notag
\end{align}
Then the additive character
$\theta: \mathbb{F}_q \rightarrow \mathbb{Q}_p(\zeta_p)$ is defined by
\begin{align}
\theta(\alpha)=\zeta_p^{\text{tr}(\alpha)}.\notag
\end{align}
For $\chi \in \widehat{\mathbb{F}_q^\times}$, the \emph{Gauss sum} is defined by
\begin{align}
G(\chi):=\sum_{x\in \mathbb{F}_q}\chi(x)\theta(x).\notag
\end{align}
We let $T$ denote a fixed generator of $\widehat{\mathbb{F}_q^\times}$ and denote by $G_m$ the Gauss sum $G(T^m)$.
We now state three results on Gauss sums which will be used to prove our main results.
\begin{lemma}\emph{(\cite[Eqn. 1.12]{greene}).}\label{fusi3}
If $k\in\mathbb{Z}$ and $T^k\neq\varepsilon$, then
$$G_kG_{-k}=qT^k(-1).$$
\end{lemma}
\begin{lemma}\emph{(\cite[Lemma 2.2]{Fuselier}).}\label{lemma1}
For all $\alpha \in \mathbb{F}_q^{\times}$, $$\theta(\alpha)=\frac{1}{q-1}\sum_{m=0}^{q-2}G_{-m}T^m(\alpha).$$
\end{lemma}
\begin{theorem}\emph{(\cite[Davenport-Hasse Relation]{Lang}).}\label{lemma3}
Let $m$ be a positive integer and let $q=p^r$ be a prime power such that $q\equiv 1 \pmod{m}$. For multiplicative characters
$\chi, \psi \in \widehat{\mathbb{F}_q^\times}$, we have
\begin{align}
\prod_{\chi^m=1}G(\chi \psi)=-G(\psi^m)\psi(m^{-m})\prod_{\chi^m=1}G(\chi).
\end{align}
\end{theorem}
\par
In the proof of our results, the Gross-Koblitz formula plays an important role.
It relates the Gauss sums and the $p$-adic gamma function.
For $n \in\mathbb{Z}^+$,
the $p$-adic gamma function $\Gamma_p(n)$ is defined as
\begin{align}
\Gamma_p(n):=(-1)^n\prod_{0<j<n,p\nmid j}j\notag
\end{align}
and one extends it to all $x\in\mathbb{Z}_p$ by setting $\Gamma_p(0):=1$ and
\begin{align}
\Gamma_p(x):=\lim_{n\rightarrow x}\Gamma_p(n)\notag
\end{align}
for $x\neq0$, where $n$ runs through any sequence of positive integers $p$-adically approaching $x$.
This limit exists, is independent of how $n$ approaches $x$,
and determines a continuous function on $\mathbb{Z}_p$ with values in $\mathbb{Z}_p^{\times}$.
\par
Let $\pi \in \mathbb{C}_p$ be the fixed root of $x^{p-1} + p=0$ which satisfies
$\pi \equiv \zeta_p-1 \pmod{(\zeta_p-1)^2}$. Then the Gross-Koblitz formula relates Gauss sums and $p$-adic gamma function as follows.
Recall that $\omega$ denotes the Teichm\"{u}ller character of $\mathbb{F}_q$.
\begin{theorem}\emph{(\cite[Gross-Koblitz]{gross}).}\label{thm4} For $a\in \mathbb{Z}$ and $q=p^r$,
\begin{align}
G(\overline{\omega}^a)=-\pi^{(p-1)\sum_{i=0}^{r-1}\langle\frac{ap^i}{q-1} \rangle}\prod_{i=0}^{r-1}\Gamma_p\left(\langle \frac{ap^i}{q-1} \rangle\right).\notag
\end{align}
\end{theorem}
\section{Proof of the results}
\par We first state a lemma which we will use to prove the main results. This lemma is a generalization of Lemma 4.1 in \cite{mccarthy2}.
For a proof, see \cite{BS1}.
\begin{lemma}\emph{(\cite[Lemma 3.1]{BS1}).}\label{lemma4}
Let $p$ be a prime and $q=p^r$. For $0\leq j\leq q-2$ and $t\in \mathbb{Z^+}$ with $p\nmid t$, we have
\begin{align}\label{eq8}
\omega(t^{tj})\prod_{i=0}^{r-1}\Gamma_p\left(\langle \frac{tp^ij}{q-1}\rangle\right)
\prod_{h=1}^{t-1}\Gamma_p\left(\langle\frac{hp^i}{t}\rangle\right)
=\prod_{i=0}^{r-1}\prod_{h=0}^{t-1}\Gamma_p\left(\langle\frac{p^ih}{t}+\frac{p^ij}{q-1}\rangle\right)
\end{align}
and
\begin{align}\label{eq9}
\omega(t^{-tj})\prod_{i=0}^{r-1}\Gamma_p\left(\langle\frac{-tp^ij}{q-1}\rangle\right)
\prod_{h=1}^{t-1}\Gamma_p\left(\langle \frac{hp^i}{t}\rangle\right)
=\prod_{i=0}^{r-1}\prod_{h=0}^{t-1}\Gamma_p\left(\langle\frac{p^i(1+h)}{t}-\frac{p^ij}{q-1}\rangle \right).
\end{align}
\end{lemma}
\begin{lemma}\label{lemma5}
For $1\leq l\leq q-2$ such that $l\neq \frac{q-1}{2}$, and $0\leq i\leq r-1$, we have
\begin{align}\label{eq-51}
&\lfloor\frac{3lp^i}{q-1}\rfloor+3\lfloor\frac{-lp^i}{q-1}\rfloor-
3\lfloor\frac{-2lp^i}{q-1}\rfloor-\lfloor\frac{6lp^i}{q-1}\rfloor\notag\\
&=-2\lfloor\langle \frac{p^i}{2}\rangle- \frac{lp^i}{q-1}\rfloor
-\lfloor\langle \frac{-p^i}{6} \rangle+ \frac{lp^i}{q-1}\rfloor-\lfloor\langle
\frac{-5p^i}{6} \rangle+\frac{lp^i}{q-1}\rfloor.
\end{align}
\end{lemma}
\begin{proof}
Since $\lfloor\frac{6lp^i}{q-1}\rfloor$ can be written as $6u+v$, for some $u,v \in \mathbb{Z}$ such that $0\leq v\leq 5$,
\eqref{eq-51} can be verified by considering the cases $v=0, 1, \ldots, 5$.
For the case $v=0$ we have $\lfloor\frac{6lp^i}{q-1}\rfloor=6u$, and then it is easy to check that both the sides of
\eqref{eq-51} are equal to zero. Similarly, for other values of $v$ one can verify the result.
\end{proof}
To prove Theorem \ref{MT1}, we will first express the number of points on the Hessian form of elliptic curves. Let $a\in \mathbb{F}_q$
be such that $a^3\neq 1$. Then the Hessian curve over $\mathbb{F}_q$ is given by the cubic equation
\begin{align}\label{hessian1}
 C_a: x^3+y^3+1=3axy.
\end{align}
We express the number of $\mathbb{F}_q$-points on $C_a$ in the following theorem.
Let $C_a(\mathbb{F}_q)=\{(x, y)\in \mathbb{F}_q^2: x^3+y^3+1=3axy\}$ be the set of all $\mathbb{F}_q$-points on $C_a$.
\begin{theorem}\label{hessian2}
Let $q=p^r, p > 5$. Then
\begin{align}
\#C_a(\mathbb{F}_q)=\alpha-1+q-q\phi(-3a)\cdot{_2}G_2\left[ \begin{array}{cc}
              \frac{1}{2}, & \frac{1}{2} \\
              \frac{1}{6}, & \frac{5}{6}
            \end{array}|\dfrac{1}{a^3}
 \right]_q,\notag
\end{align}
where
$\alpha=\left\{
           \begin{array}{ll}
             5-6\phi(-3), & \hbox{if~ $q\equiv 1\pmod{3}$;} \\
             1, & \hbox{if~ $q\not\equiv 1\pmod{3}$.}
           \end{array}
         \right.$
\end{theorem}
\begin{proof}
We have
$\#C_{a}(\mathbb{F}_{q})=\#\{(x,y)\in\mathbb{F}_{q}\times\mathbb{F}_{q}:\ P(x,y)=0\}$,\\
where $P(x,y)=x^{3}+y^{3}-3axy+1$. Using the identity
\begin{align}
\sum_{z\in\mathbb{F}_q}\theta(zP(x,y))=\left\{
                                         \begin{array}{ll}
                                           q, & \hbox{if $P(x,y)=0$;} \\
                                           0, & \hbox{if $P(x,y)\neq0$,}
                                         \end{array}
                                       \right.\notag
\end{align}
we obtain
\begin{align}\label{eq-53}
q\cdot\#C_{a}(\mathbb{F}_{q})&=\sum_{x,y,z\in\mathbb{F}_{q}}\theta(zP(x,y))\notag\\
&=q^2+\sum_{z\in\mathbb{F}_{q}^{\times}}\theta(z)+\sum_{y,z\in\mathbb{F}_{q}^{\times}}\theta(zy^3)\theta(z)\notag\\
&~+\sum_{x,z\in\mathbb{F}_{q}^{\times}}\theta(zx^3)\theta(z)+\sum_{x,y,z\in\mathbb{F}_{q}^{\times}}\theta(z)\theta(zx^3)
\theta(zy^3)\theta(-3axyz)\notag\\
&:=q^2+A+B+C+D.
\end{align}
Using Lemma \ref{lemma2}, Lemma \ref{fusi3} and Lemma \ref{lemma1}, we find $A$, $B$, $C$ and $D$ separately. We have
\begin{align}
A=\frac{1}{q-1}\sum_{l=0}^{q-2}G_{-l}\sum_{z\in\mathbb{F}_{q}^{\times}}T^{l}(z).\notag
\end{align}
The inner sum in the expression of $A$ is non zero only if $l=0$, and hence $A=-1$. We have
\begin{align}
B&=\sum_{y,z\in\mathbb{F}_{q}^{\times}}\theta(zy^3)\theta(z)\notag\\
&=\frac{1}{(q-1)^2}\sum_{y,z\in\mathbb{F}_{q}^{\times}}\sum_{l,m=0}^{q-2}G_{-m}T^m(zy^3)G_{-l}T^l(z)\notag\\
&=\frac{1}{(q-1)^2}\sum_{l,m=0}^{q-2}G_{-m}G_{-l}\sum_{z\in\mathbb{F}_{q}^{\times}}T^{l+m}(z)
\sum_{y\in\mathbb{F}_{q}^{\times}}T^{3m}(y),\notag
\end{align}
which is non zero only if $l+m=0$ and $3m=0$.
By considering the following two cases we find $B$.\\
Case 1: If $q\equiv 1\pmod{3}$ then $m=0,\frac{q-1}{3}$ or $\frac{2(q-1)}{3}$; and $l=0,-\frac{q-1}{3}$ or $-\frac{2(q-1)}{3}$. Hence,
\begin{align}
B&=G_{0}G_{0}+G_{-\frac{q-1}{3}}G_{\frac{q-1}{3}}+G_{-\frac{2(q-1)}{3}}G_{\frac{2(q-1)}{3}}\notag\\
&=1+2q.\notag
\end{align}
Case 2: If $q\not\equiv 1\pmod{3}$ then $l=m=0$, and hence $B=G_{0}G_{0}=1$. Also,
\begin{align}
C&=\sum_{x,z\in\mathbb{F}_{q}^{\times}}\theta(zx^3)\theta(z)\notag\\
&=B.\notag
\end{align}
Finally,
\begin{align}
D&=\sum_{x,y,z\in\mathbb{F}_{q}^{\times}}\theta(z)\theta(zx^3)\theta(zy^3)\theta(-3axyz)\notag\\
&=\frac{1}{(q-1)^4}\sum_{x,y,z\in\mathbb{F}_{q}^{\times}}\sum_{l,m,n,k=0}^{q-2}G_{-l}G_{-m}G_{-n}G_{-k}T^{l}(zx^3)\notag\\
&~\times T^{m}(zy^3)T^{n}(z)T^{k}(-3axyz)\notag\\
&=\frac{1}{(q-1)^4}\sum_{l,m,n,k=0}^{q-2}G_{-l}G_{-m}G_{-n}G_{-k}T^{k}(-3a)\notag\\
&\times~\sum_{x\in\mathbb{F}_{q}^{\times}}T^{3l+k}(x)\sum_{y\in\mathbb{F}_{q}^{\times}}T^{3m+k}(y)
\sum_{z\in\mathbb{F}_{q}^{\times}}T^{l+m+n+k}(z),\notag
\end{align}
which is non zero only if $3l+k=0$, $3m+k=0$, and $l+m+n+k=0$.
We now consider the following two cases.\\
Case 1: If $q\equiv 1\pmod{3}$ then $m=l$, $l+\frac{q-1}{3}$ or $l+\frac{2(q-1)}{3}$; $k=-3l$;
and $n=l$, $l-\frac{q-1}{3}$ or $l-\frac{2(q-1)}{3}$, and hence
\begin{align}\label{eq-54}
D&=\frac{1}{q-1}\sum_{l=0}^{q-2}G_{-l}G_{-l}G_{-l}G_{3l}T^{-3l}(-3a)\notag\\
&~+\frac{2}{q-1}\sum_{l=0}^{q-2}G_{-l}G_{-l-\frac{q-1}{3}}G_{-l-\frac{2(q-1)}{3}}G_{3l}T^{-3l}(-3a).
\end{align}
Transforming $l\rightarrow l-\frac{q-1}{2}$, we have
\begin{align}
D&=\frac{1}{q-1}\sum_{l=0}^{q-2}G_{-l+\frac{q-1}{2}}G_{-l+\frac{q-1}{2}}G_{-l+\frac{q-1}{2}}G_{3l-\frac{q-1}{2}}
T^{-3l+\frac{q-1}{2}}(-3a)\notag\\
&~+\frac{2}{q-1}\sum_{l=0}^{q-2}G_{-l+\frac{q-1}{2}}G_{-l+\frac{q-1}{6}}G_{-l-\frac{q-1}{6}}G_{3l-\frac{q-1}{2}}
T^{-3l+\frac{q-1}{2}}(-3a)\notag\\
&=\frac{\phi(-3a)}{q-1}\sum_{l=0}^{q-2}G_{-l+\frac{q-1}{2}}G_{-l+\frac{q-1}{2}}G_{-l+\frac{q-1}{2}}G_{3l-\frac{q-1}{2}}
T^{-3l}(-3a)\notag\\
&~+\frac{2\phi(-3a)}{q-1}\sum_{l=0}^{q-2}G_{-l+\frac{q-1}{2}}G_{-l+\frac{q-1}{6}}G_{-l-\frac{q-1}{6}}G_{3l-\frac{q-1}{2}}
T^{-3l}(-3a).\notag
\end{align}
Using Davenport-Hasse relation for certain values of $m$ and $\psi$ we deduce the following relations:
For $m=2$, $\psi=T^{-l}$, we have
\begin{align}
G_{-l+\frac{q-1}{2}}=\frac{G_{\frac{q-1}{2}}G_{-2l}T^l(4)}{G_{-l}},\notag
\end{align}
and for $m=2$, $\psi=T^{3l}$, we have
\begin{align}
G_{3l-\frac{q-1}{2}}=\frac{G_{\frac{q-1}{2}}G_{6l}T^{-3l}(4)}{G_{3l}}.\notag
\end{align}
For $m=6$, $\psi=T^{-l}$, we have
\begin{align}
&G_{-l+\frac{q-1}{2}}G_{-l+\frac{q-1}{3}}G_{-l+\frac{2(q-1)}{3}}G_{-l+\frac{q-1}{6}}G_{-l+\frac{5(q-1)}{6}}\notag\\
&=\frac{q^2\phi(-1)G_{\frac{q-1}{2}}G_{-6l}T^{6l}(6)}{G_{-l}},\notag
\end{align}
and for $m=3$, $\psi=T^{-l}$, we have
\begin{align}
G_{-l+\frac{q-1}{3}}G_{-l+\frac{2(q-1)}{3}}=\frac{qG_{-3l}T^{3l}(3)}{G_{-l}}.\notag
\end{align}
Using all these expressions and Lemma \ref{lemma2} and Lemma \ref{fusi3} we find that
\begin{align}
D&=\frac{\phi(-3a)}{q-1}\sum_{l=0}^{q-2}\frac{G_{-2l}G_{-2l}G_{-2l}G_{6l}G_{\frac{q-1}{2}}^4T^{-3l}(-3a)}
{G_{-l}G_{-l}G_{-l}G_{3l}}\notag\\
&~+\frac{2\phi(-3a)}{q-1}\sum_{l=0}^{q-2}\frac{G_{-l+\frac{q-1}{2}}G_{-l+\frac{q-1}{3}}G_{-l+\frac{2(q-1)}{3}}
G_{-l+\frac{q-1}{6}}G_{-l+\frac{5(q-1)}{6}}G_{3l-\frac{q-1}{2}}T^{-3l}(-3a)}{G_{-l+\frac{q-1}{3}}G_{-l+\frac{2(q-1)}{3}}}\notag\\
&=\frac{q^2\phi(-3a)}{q-1}\sum_{l=0}^{q-2}\frac{G_{-2l}G_{-2l}G_{-2l}G_{6l}T^{-3l}(-3a)}{G_{-l}G_{-l}G_{-l}G_{3l}}\notag\\
&~+\frac{2q^2\phi(-3a)}{q-1}\sum_{l=0}^{q-2}\frac{G_{-6l}G_{6l}T^{-3l}(-a)}{G_{3l}G_{-3l}}\notag\\
&=\frac{q^2\phi(-3a)}{q-1}\sum_{l=0}^{q-2}\frac{G_{-2l}^3G_{6l}}{G_{-l}^3G_{3l}}T^{-3l}(-3a)+
\frac{6q^2\phi(-3a)\phi(a)}{(q-1)q}\notag\\
&~+\frac{2q^2\phi(-3a)}{q-1}\sum_{l=0, l\neq\frac{q-1}{6},\frac{q-1}{2},\frac{5(q-1)}{6}}^{q-2}T^{3l}(\frac{1}{a})\notag\\
&=\frac{q^2\phi(-3a)}{q-1}\sum_{l=0}^{q-2}\frac{G_{-2l}^3G_{6l}}{G_{-l}^3G_{3l}}T^{-3l}(-3a)+
\frac{6q\phi(-3)}{q-1}\notag\\
&~+\frac{2q^2\phi(-3a)}{q-1}\sum_{l=0}^{q-2}T^{3l}(\frac{1}{a})-\frac{6q^2\phi(-3a)\phi(a)}{q-1}\notag\\
&=\frac{q^2\phi(-3a)}{q-1}\sum_{l=0}^{q-2}\frac{G_{-2l}^3G_{6l}}{G_{-l}^3G_{3l}}T^{-3l}(-3a)-6q\phi(-3)\notag\\
&=\frac{q^2\phi(-3a)}{q-1}\sum_{l=0,l\neq\frac{q-1}{2}}^{q-2}\frac{G_{-2l}^3G_{6l}}{G_{-l}^3G_{3l}}T^{-3l}(-3a)
+\frac{1}{q-1}-6q\phi(-3).\notag
\end{align}
Taking $T=\overline{\omega}$ and using Gross-Koblitz formula we deduce that
\begin{align}
D&=\frac{q^2\phi(-3a)}{q-1}\sum_{l=0,l\neq\frac{q-1}{2}}^{q-2}\pi^{(p-1)\sum_{i=0}^{r-1}
\{3\langle\frac{-2lp^i}{q-1}\rangle+\langle\frac{6lp^i}{q-1}\rangle-\langle\frac{3lp^i}{q-1}\rangle-
3\langle\frac{-lp^i}{q-1}\rangle\}}\notag\\
&~\times\overline{\omega}^l\left(-\frac{1}{27a^3}\right)
\prod_{i=0}^{r-1}\frac{\Gamma_{p}^3(\langle\frac{-2lp^i}{q-1}\rangle)\Gamma_p(\langle\frac{6lp^i}{q-1}\rangle)}
{\Gamma_{p}^3(\langle\frac{-lp^i}{q-1}\rangle)\Gamma_p(\langle\frac{3lp^i}{q-1}\rangle)}\notag\\
&~+\frac{1}{q-1}-6q\phi(-3).\notag
\end{align}
From Lemma \ref{lemma4} we deduce that
\begin{align}
D&=\frac{q^2\phi(-3a)}{q-1}\sum_{l=0,l\neq\frac{q-1}{2}}^{q-2}\pi^{(p-1)s}~~\overline{\omega}^l\left(-\frac{1}{a^3}\right)\notag\\
&~\times\prod_{i=0}^{r-1}\left\{\frac{\Gamma_{p}^3(\langle(\frac{1}{2}-\frac{l}{q-1})p^i\rangle)
\Gamma_p(\langle(\frac{1}{6}+\frac{l}{q-1})p^i\rangle)}
{\Gamma_{p}^3(\langle\frac{p^i}{2}\rangle)\Gamma_p(\langle\frac{p^i}{6}\rangle)}\right\}\notag\\
&~\times\prod_{i=0}^{r-1}\left\{\frac{\Gamma_p(\langle(\frac{1}{2}+\frac{l}{q-1})p^i\rangle)
\Gamma_p(\langle(\frac{5}{6}+\frac{l}{q-1})p^i\rangle)}
{\Gamma_p(\langle\frac{p^i}{2}\rangle)\Gamma_p(\langle\frac{5p^i}{6}\rangle)}\right\}\notag\\
&~+\frac{1}{q-1}-6q\phi(-3),\notag
\end{align}
 where $s=\sum_{i=0}^{r-1}
\{3\langle\frac{-2lp^i}{q-1}\rangle+\langle\frac{6lp^i}{q-1}\rangle-\langle\frac{3lp^i}{q-1}\rangle-
3\langle\frac{-lp^i}{q-1}\rangle\}$.
\begin{align}
D&=\frac{q^2\phi(-3a)}{q-1}\sum_{l=0,l\neq\frac{q-1}{2}}^{q-2}\pi^{(p-1)s}~~\overline{\omega}^l\left(-\frac{1}{a^3}\right)\notag\\
&~\times\underbrace{\prod_{i=0}^{r-1}\left\{\frac{\Gamma_p(\langle(\frac{1}{2}-\frac{l}{q-1})p^i\rangle)
\Gamma_p(\langle(\frac{1}{2}+\frac{l}{q-1})p^i\rangle)}
{\Gamma_p(\langle\frac{p^i}{2}\rangle)\Gamma_p(\langle\frac{p^i}{2}\rangle)}\right\}}\\
&\hspace{3.9cm}I_{l}\notag\\
&~\times\prod_{i=0}^{r-1}\left\{\frac{\Gamma_{p}^2(\langle(\frac{1}{2}-\frac{l}{q-1})p^i\rangle)
\Gamma_p(\langle(\frac{1}{6}+\frac{l}{q-1})p^i\rangle)\Gamma_p(\langle(\frac{5}{6}+\frac{l}{q-1})p^i\rangle)}
{\Gamma_{p}^2(\langle\frac{p^i}{2}\rangle)\Gamma_p(\langle\frac{5p^i}{6}\rangle)}\right\}\notag\\
&~+\frac{1}{q-1}-6q\phi(-3).\notag
\end{align}
For $l\neq \frac{q-1}{2}$, we have
\begin{align}\label{eq-55}
I_{l}&=\prod_{i=0}^{r-1}\frac{\Gamma_p(\langle(\frac{1}{2}-\frac{l}{q-1})p^i\rangle)
\Gamma_p(\langle(\frac{1}{2}+\frac{l}{q-1})p^i\rangle)}
{\Gamma_p(\langle\frac{p^i}{2}\rangle)\Gamma_p(\langle\frac{p^i}{2}\rangle)}\notag\\
&=\prod_{i=0}^{r-1}\frac{\Gamma_p(\langle(\frac{1}{2}-\frac{l}{q-1})p^i\rangle)
\Gamma_p(\langle(1-\frac{l}{q-1})p^i\rangle)\Gamma_p(\langle\frac{lp^i}{q-1}\rangle)
\Gamma_p(\langle(\frac{1}{2}+\frac{l}{q-1})p^i\rangle)}
{\Gamma_p(\langle\frac{p^i}{2}\rangle)\Gamma_p(\langle\frac{p^i}{2}\rangle)}\notag\\
&\times\frac{1}{\Gamma_p(\langle(1-\frac{l}{q-1})p^i\rangle)
\Gamma_p(\langle\frac{lp^i}{q-1}\rangle)}.
\end{align}
Applying Lemma \ref{lemma4} in equation \eqref{eq-55} we deduce that
\begin{align}\label{eq-56}
I_l&=\prod_{i=0}^{r-1}\frac{\Gamma_p(\langle\frac{-2lp^i}{q-1}\rangle)\Gamma_p(\langle\frac{2lp^i}{q-1}\rangle)}
{\Gamma_p(\langle(1-\frac{l}{q-1})p^i\rangle)
\Gamma_p(\langle\frac{lp^i}{q-1}\rangle)}.
\end{align}
From \cite[Eqn. 2.9]{mccarthy2} we have that for $0<l<q-1$,
$$\prod_{i=0}^{r-1}\Gamma_p(\langle(1-\frac{l}{q-1})p^i\rangle)
\Gamma_p(\langle\frac{lp^i}{q-1}\rangle)=(-1)^r\overline{\omega}^l(-1).$$
Putting this value in equation \eqref{eq-56}, and using Gross-Koblitz formula [Theorem \ref{thm4}], Lemma \ref{fusi3},
and the fact that $$\langle\frac{-2lp^i}{q-1}\rangle+\langle\frac{2lp^i}{q-1}\rangle=1,$$
we have
\begin{align}
I_l&=\frac{\pi^{(p-1)\sum_{i=0}^{r-1}\langle\frac{-2lp^i}{q-1}\rangle}\prod_{i=0}^{r-1}\Gamma_p\left(\langle\frac{-2lp^i}{q-1}
\rangle\right)
\pi^{(p-1)\sum_{i=0}^{r-1}\langle\frac{2lp^i}{q-1}\rangle}\prod_{i=0}^{r-1}\Gamma_p\left(\langle\frac{2lp^i}{q-1}\rangle\right)}
{(-1)^r\overline{\omega}^l(-1)\pi^{(p-1)\sum_{i=0}^{r-1}\{\langle\frac{-2lp^i}{q-1}\rangle+\langle\frac{2lp^i}{q-1}\rangle\}}}
\notag\\
&=\frac{G(\overline{\omega}^{~-2l})G(\overline{\omega}^{~2l})}{q\overline{\omega}^l(-1)}\notag\\
&=\frac{q~\overline{\omega}^{2l}(-1)}{q~\overline{\omega}^{l}(-1)}\notag\\
&=\overline{\omega}^{l}(-1).\notag
\end{align}
Using the above relation we obtain
\begin{align}\label{eq-52}
D&=\frac{q^2\phi(-3a)}{q-1}\sum_{l=0,l\neq\frac{q-1}{2}}^{q-2}\pi^{(p-1)s}~~
\overline{\omega}^l\left(\frac{1}{a^3}\right)\notag\\
&~\times\prod_{i=0}^{r-1}\left\{\frac{\Gamma_{p}^2(\langle(\frac{1}{2}-\frac{l}{q-1})p^i\rangle)
\Gamma_p(\langle(\frac{1}{6}+\frac{l}{q-1})p^i\rangle)\Gamma_p(\langle(\frac{5}{6}+\frac{l}{q-1})p^i\rangle)}
{\Gamma_{p}^2(\langle\frac{p^i}{2}\rangle)\Gamma_p(\langle\frac{5p^i}{6}\rangle)}\right\}\notag\\
&~+\frac{1}{q-1}-6q\phi(-3).
\end{align}
Now
\begin{align}
s&=\sum_{i=0}^{r-1}
\{3\langle\frac{-2lp^i}{q-1}\rangle+\langle\frac{6lp^i}{q-1}\rangle-\langle\frac{3lp^i}{q-1}\rangle-
3\langle\frac{-lp^i}{q-1}\rangle\}\notag\\
&=\sum_{i=0}^{r-1}
\{3(\frac{-2lp^i}{q-1})+(\frac{6lp^i}{q-1})-(\frac{3lp^i}{q-1})-
3(\frac{-lp^i}{q-1})\}\notag\\
&~+\sum_{i=0}^{r-1}
\{-3\lfloor\frac{-2lp^i}{q-1}\rfloor-\lfloor\frac{6lp^i}{q-1}\rfloor+\lfloor\frac{3lp^i}{q-1}\rfloor+
3\lfloor\frac{-lp^i}{q-1}\rfloor\}\notag\\
&=\sum_{i=0}^{r-1}
\{-3\lfloor\frac{-2lp^i}{q-1}\rfloor-\lfloor\frac{6lp^i}{q-1}\rfloor+\lfloor\frac{3lp^i}{q-1}\rfloor+
3\lfloor\frac{-lp^i}{q-1}\rfloor\},\notag
\end{align}
which is an integer.
Therefore equation \eqref{eq-52} becomes
\begin{align}
D&=\frac{q^2\phi(-3a)}{q-1}\sum_{l=0,l\neq\frac{q-1}{2}}^{q-2}(-p)^s~~
\overline{\omega}^l\left(\frac{1}{a^3}\right)\notag\\
&~\times\prod_{i=0}^{r-1}\left\{\frac{\Gamma_{p}^2(\langle(\frac{1}{2}-\frac{l}{q-1})p^i\rangle)
\Gamma_p(\langle(\frac{1}{6}+\frac{l}{q-1})p^i\rangle)\Gamma_p(\langle(\frac{5}{6}+\frac{l}{q-1})p^i\rangle)}
{\Gamma_{p}^2(\langle\frac{p^i}{2}\rangle)\Gamma_p(\langle\frac{5p^i}{6}\rangle)}\right\}\notag\\
&~+\frac{1}{q-1}-6q\phi(-3).
\end{align}
Lemma \ref{lemma5} gives
\begin{align}
D&=\frac{q^2\phi(-3a)}{q-1}\sum_{l=0,l\neq\frac{q-1}{2}}^{q-2}\overline{\omega}^l\left(\frac{1}{a^3}\right)
(-p)^{\sum_{i=0}^{r-1}\{-2\lfloor\langle\frac{p^i}{2}\rangle-\frac{lp^i}{q-1}\rfloor\}}\notag\\
&~\times(-p)^{\sum_{i=0}^{r-1}\{-\lfloor\langle\frac{-p^i}{6}\rangle+\frac{lp^i}{q-1}\rfloor
-\lfloor\langle\frac{-5p^i}{6}\rangle+\frac{lp^i}{q-1}\rfloor\}}\notag\\
&~\times\prod_{i=0}^{r-1}\left\{\frac{\Gamma_{p}^2(\langle(\frac{1}{2}-\frac{l}{q-1})p^i\rangle)
\Gamma_p(\langle(\frac{1}{6}+\frac{l}{q-1})p^i\rangle)\Gamma_p(\langle(\frac{5}{6}+\frac{l}{q-1})p^i\rangle)}
{\Gamma_{p}^2(\langle\frac{p^i}{2}\rangle)\Gamma_p(\langle\frac{5p^i}{6}\rangle)}\right\}\notag\\
&~+\frac{1}{q-1}-6q\phi(-3)\notag\\
&=\frac{q^2\phi(-3a)}{q-1}\sum_{l=0}^{q-2}\overline{\omega}^l\left(\frac{1}{a^3}\right)
(-p)^{\sum_{i=0}^{r-1}\{-2\lfloor\langle\frac{p^i}{2}\rangle-\frac{lp^i}{q-1}\rfloor\}}\notag\\
&~\times(-p)^{\sum_{i=0}^{r-1}\{-\lfloor\langle\frac{-p^i}{6}\rangle+\frac{lp^i}{q-1}\rfloor
-\lfloor\langle\frac{-5p^i}{6}\rangle+\frac{lp^i}{q-1}\rfloor\}}\notag\\
&~\times\prod_{i=0}^{r-1}\left\{\frac{\Gamma_{p}^2(\langle(\frac{1}{2}-\frac{l}{q-1})p^i\rangle)
\Gamma_p(\langle(\frac{1}{6}+\frac{l}{q-1})p^i\rangle)\Gamma_p(\langle(\frac{5}{6}+\frac{l}{q-1})p^i\rangle)}
{\Gamma_{p}^2(\langle\frac{p^i}{2}\rangle)\Gamma_p(\langle\frac{5p^i}{6}\rangle)}\right\}\notag\\
&~-\frac{q}{q-1}+\frac{1}{q-1}-6q\phi(-3)\notag\\
&=-1-6q\phi(-3)-q^2\phi(-3a)\cdot {_2}G_{2}\left[\begin{array}{cc}
                                               \frac{1}{2} & \frac{1}{2} \\
                                               \frac{1}{6} & \frac{5}{6}
                                             \end{array}|\frac{1}{a^3}
\right]_{q}.
\end{align}
Case 2: If $q\not\equiv 1 \pmod{3}$ then
$m=l$, $k=-3l$, and $n=l$, and then
\begin{align}
D&=\frac{1}{q-1}\sum_{l=0}^{q-2}G_{-l}G_{-l}G_{-l}G_{3l}T^{-3l}(-3a),\notag
\end{align}
which is same as the first term of the equation \eqref{eq-54}.
Thus we have
$$D=-1-q^2\phi(-3a)\cdot {_2}G_{2}\left[\begin{array}{cc}
                                               \frac{1}{2} & \frac{1}{2} \\
                                               \frac{1}{6} & \frac{5}{6}
                                             \end{array}|\frac{1}{a^3}
\right]_{q}.$$
Substituting the values of $A$, $B$, $C$ and $D$ in equation \eqref{eq-53} we obtain the desired result.
\end{proof}
\noindent \textbf{Proof of Theorem \ref{MT1}}:
Consider the elliptic curve $$E: y^2=x^3+mx+n,$$ where $m=-27d(d^3+8)$ and $n=27(d^6-20d^3-8)$. By the
following transformation
$x\rightarrow -\frac{36-9d^3+3dx-y}{6(9d^2+x)}$ and $y\rightarrow -\frac{36-9d^3+3dx+y}{6(9d^2+x)}$,
we obtain the equivalent form $C_d$. In the proof of Theorem 1.2, Barman and Kalita \cite{BK1} proved that
$$\#E(\mathbb{F}_q)+q=\#C_d(\mathbb{F}_q)+2+\phi(-3(8+92d^3+35d^6)).$$
For $d^3\neq 1$, from Theorem \ref{hessian2}, we have
\begin{align}
 a_q(E)&=q+1-\#E(\mathbb{F}_q)\notag\\
 &=2q-1-\#C_d(\mathbb{F}_q)-\phi(-3(8+92d^3+35d^6))\notag\\
 &=q-\alpha-\phi(-3(8+92d^3+35d^6))+q\phi(-3d)\cdot{_2}G_2\left[ \begin{array}{cc}
              \frac{1}{2}, & \frac{1}{2} \\
              \frac{1}{6}, & \frac{5}{6}
            \end{array}|\dfrac{1}{d^3}
 \right]_q,\notag
\end{align}
where
$\alpha=\left\{
           \begin{array}{ll}
             5-6\phi(-3), & \hbox{if~ $q\equiv 1\pmod{3}$;} \\
             1, & \hbox{if~ $q\not\equiv 1\pmod{3}$.}
           \end{array}
         \right.$\\
For $m, n \neq 0$ and $-\dfrac{27n^2}{4m^3}\neq 1$, we have $j(E)\neq 0, 1728$.
Now, applying Theorem \ref{mc} over $\mathbb{F}_q$, we complete the proof of the theorem.

\bibliographystyle{amsplain}

\begin{thebibliography}{10}
\bibitem{BK1}
R. Barman and G. Kalita, {\it Elliptic curves and special values of Gaussian hypergeometric series},
J. Number Theory 133 (2013), 3099--3111.

\bibitem{BK2}
R. Barman and G. Kalita, {\it Hypergeometric functions over $\mathbb{F}_q$ and traces of Frobenius for elliptic curves},
Proc. Amer. Math. Soc. 141 (2013), 3403--3410.

\bibitem{BS1}
R. Barman and N. Saikia, {\it $p$-Adic gamma function and the trace of Frobenius of elliptic curves},
J. Number Theory (to appear).

\bibitem{evans}
B. Berndt, R. Evans, and K. Williams, {\it Gauss and Jacobi Sums}, Canadian Mathematical Society Series of Monographs and Advanced Texts,
A Wiley-Interscience Publication, John Wiley \& Sons, Inc., New York, 1998.

\bibitem {Fuselier} J. Fuselier, \textit{Hypergeometric functions over $\mathbb{F}_p$ and relations to elliptic curve and modular forms},
Proc. Amer. Math. Soc. 138 (2010), 109--123.

\bibitem{greene}
J. Greene, {\it Hypergeometric functions over finite fields}, Trans. Amer. Math. Soc. 301 (1) (1987), 77--101.

\bibitem{gross}
B. H. Gross and N. Koblitz, {\it Gauss sum and the $p$-adic $\Gamma$-function}, Annals of Mathematics 109 (1079), 569--581.

\bibitem{ireland}
K. Ireland and M. Rosen, {\it A Classical Inroduction to Modern Number Theory}, Springer International Edition, Springer, 2005.

\bibitem {Lang} S. Lang, \textit{Cyclotomic Fields I and II},
Graduate Texts in Mathematics, vol. 121, Springer-Verlag, New York, 1990.

\bibitem {lennon} C. Lennon, \textit{Gaussian hypergeometric evaluations of traces of Frobenius for elliptic curves},
Proc. Amer. Math. Soc. 139 (2011), 1931--1938.

\bibitem {lennon2} C. Lennon, \textit{Trace formulas for Hecke operators, Gaussian hypergeometric functions, and the modularity of a threefold},
J. Number Theory 131 (12) (2011), 2320--2351.

\bibitem{mccarthy2}
D. McCarthy, {\it The trace of Frobenius of elliptic curves and the $p$-adic gamma function}, Pacific J. Math. 261 (1) (2013), 219--236.

\bibitem{mccarthy3}
D. McCarthy, {\it Extending Gaussian hypergeometric series to the $p$-adic setting}, Int. J. Number Theory 8 (7) (2012), 1581--1612.

\end{thebibliography}
%  Insert the bibliography data here.

\end{document}